\documentclass{amsart}
\usepackage{tikz}
\usetikzlibrary{shapes.geometric}
\usepackage{amsfonts, amsmath, amssymb, amsthm}
\newtheorem{theorem}{Theorem}[section]
\newtheorem{remark}[theorem]{ Remark}

\newtheorem{example}[theorem]{ Example}

\newtheorem{proposition}[theorem]{Proposition}

\newtheorem{lemma}[theorem]{Lemma}

\newtheorem{definition}[theorem]{Definition}

\newcommand{\beq} {\begin{equation}}
\newcommand{\eeq} {\end{equation}}

\begin{document}

\title[Zero-Set Intersection Graph on $C_{+}(X)$]{ Zero-Set Intersection Graph on $C_{+}(X)$}
\author{Soumi Basu, Bedanta Bose}

\address{Jadavpur University, Department of Mathematics, Kolkata, West Bengal, India}
\email{basu.soumi2018@gmail.com}

\address{Swami Niswambalananda Girls' College\\ 115, BPMB Sarani\\ Bhadrakali, Hooghly - 712232\\ India}
\email{anabedanta@gmail.com}

\begin{abstract}
	For any Tychonoff space $X$ we have introduced the zero-set intersection graph on $\Gamma(C_{+}(X))$ and studied the graph properties in connection with the algebraic properties of the semiring $C_{+}(X)$. We have shown that for any two realcompact spaces $X$ and $Y$ the graph isomorphism between $\Gamma(C_{+}(X))$ and $\Gamma(C_{+}(Y))$, the semiring isomorphism between $C_{+}(X)$ and $C_{+}(Y)$, the topological homeomorphism between $X$ and $Y$, the ring isomorphism between $C(X)$ and $C(Y)$ and the graph isomorphism between $\Gamma(C(X))$ and $\Gamma(C(Y))$ are equivalent.
	\end{abstract}
\maketitle

\section{Introduction}

 In 1993, the structure of $C_{+}(X)$ was first studied by S. K. Acharyya, K. C. Chattopadhyay and G. G. Ray \cite{Acharyya1}. The collection $C_{+}(X)$ of all non-negative real valued continuous functions over a topological space $X$ forms a semiring with respect to pointwise addition and multiplication of functions (Roughly speaking, semirings are rings without the requirement of the additive inverses). 
 In 2019, E. M. Vechtomov {\it et. al.} in their paper \cite{Vechtomov}, reviewed all the results intensively studied by them and many other authors on the theory of semirings of continuous functions, which exhibit sustained research interest in this structure.
All such study on $C_{+}(X)$ is totally focused on two types of findings: one is the effect of the topological properties of $X$ on the algebraic structure of $C_{+}(X)$ and vice-versa. Many aspects of these dual study has helped to dig into the deeper area of research and we refer to \cite{Acharyya1}, \cite{Ray 1995}, \cite{Vechtomov} for one who wants to study further on this topic. One question which we think important, have not been nurtured during the course of the above mentioned study, is that whether the study of the algebraic structures is the only approach to characterize the corresponding topological spaces and vice-versa. To investigate this, we introduce the graph structure in $C_{+}(X)$ with an intention to study its connection with the algebraic structure of the semiring $C_{+}(X)$ and topological structure of $X$. 

Study of the graph structure on rings is not at all a new area of research. Beck introduced \cite{beck} the idea of zero divisor graph of a commutative ring with unity and later on many research, for instance \cite{anderson-livingston, ref1-1, ref1-7, ref1-4, ref1-5, ref1-3, ref1-10, ref1-11, ref1-13, ref1-14, ref1-15, ref1-16, ref1-17, ref1-21}, has been done in this area. In case of the ring of real valued continuous functions $C(X)$ on a topological space $X$, Azarpanah {\it et. al.} \cite{zero-divisor-c(x)} studied the zero divisor graph of $C(X)$. They obtained the conditions on $X$ under which the associated graph is triangulated, connected etc. However, as their work was a follow up of zero divisor graph of a ring, the main topological ingredients of characterizing the graph was missing. Apart from that Amini {\it et. al.} \cite{graph-ideal} also studied another graph structure on $C(X)$ from the co-maximal ideal point of view. Later in \cite{ref1}, B. Bose and A. Das introduced the zero-set intersection graph $\Gamma (C(X))$ on $C(X)$ with a view to associate the topological properties of $X$, the algebraic properties of $C(X)$ and the graph properties of $\Gamma (C(X))$. Recently, Acharyya {\it et. al.} studied the Zero-Divisor graph of the rings $C^{\mathcal{P}}(X)$ and $C^{\mathcal{P}}_{\infty}(X)$ \cite{Acharyya zero-divisor}. But the study of the graph on the semiring $C_{+}(X)$ where the algebraic structures are lesser than $C(X)$ is quite new as of our knowledge. In this paper we initiate the study of the `Zero-Set intersection graph' structure on $C_{+}(X)$ with a goal to characterize some topological properties of $X$ and semiring properties of $C_{+}(X)$.

At first in section 2, we recall some definitions and results related to graphs, rings and semirings. In section 3, we define the zero-set intersection graph of $C_{+}(X)$ and study few graph properties. Then we study the cliques, maximal cliques and prime cliques of $\Gamma(C_{+}(X))$ and their relations with the ideals, maximal ideals, prime ideals of $C_{+}(X)$ respectively. We show that the maximal cliques of $\Gamma(C_{+}(X))$ can be obtained from the maximal cliques of $\Gamma(C(X))$ and establish their connections with those of the graph $\Gamma(C(X))$. The precise form of maximal cliques have been obtained (see Theorem \ref{max2}). Results connecting maximal (prime) ideals of $C_{+}(X)$ to maximal (respectively, prime) cliques of $\Gamma(C_{+}(X))$ have been proved (see Theorems \ref{max3}, \ref{max5}, \ref{max6}, \ref{max4}, \ref{prime2}) analogous to those proved in the context of the graph $\Gamma(C(X))$ of $C(X)$ in \cite{ref1}. Also the neighbourhood properties of $\Gamma(C_{+}(X))$ are studied and it has been shown that the neighbourhood properties are preserved under the graph isomorphism (see Theorem \ref{4.2}). Finally, in section 4, the inter-relationships between graph isomorphisms, ring isomorphism, semiring isomorphisms and homeomorphisms of spaces have been studied and the desired equivalence of those has been established (see Theorem \ref{combined main result}). We obtain that the space $X$ is homeomorphic to the space $Y$ if and only if the graphs $\Gamma(C_{+}(X))$ and $\Gamma(C_{+}(Y))$ are isomorphic, which says that the graph structure on $C_{+}(X)$ essentially makes distinction between the spaces belonging to the class of all Hewitt spaces. For proving the equivalence in Theorem \ref{combined main result}, the result that plays a key role is Theorem \ref{small graph to big}, which tells us that for topological spaces $X$ and $Y$, if the subgraphs $\Gamma(C_{+}(X))$ and $\Gamma(C_{+}(Y))$ are isomorphic then the graphs $\Gamma(C(X))$ and $\Gamma(C(Y))$ are isomorphic. Also a combined picture of main results of section 4 is given.

\section{Preliminaries}

For clarity and make the paper self sufficient, we want to recall some preliminary facts about graph theory and semiring that will be used in the sequel.

For a non-empty set $V$ and a symmetric binary relation (possibly empty) $E$ on $V$, $G = (V,E)$ is called a \textit{graph}. The set $V$ is called the set of vertices and $E$ is called the set of edges of $G$. Two element $u$ and $v$ in $V$ are said to be \textit{adjacent} if $(u,v)\in E$. $H = (W, F)$ is called a \textit{subgraph} of $G$ if $H$ itself is a graph and $\phi\neq W\subseteq V$ and $F\subseteq E$. $G$ is said to be \textit{complete} if all the vertices of $G$ are pairwise adjacent. A complete subgraph of a graph $G$ is called a \textit{clique}. A clique which is maximal with respect to inclusion is called a \textit{maximal clique}. Two graphs $G = (V,E)$ and $G'=(V',E')$ are said to be \textit{isomorphic} if there exists a bijection $\phi: V \rightarrow V'$ such that $(u,v)\in E$ if and only if $(\phi(u), \phi(v))\in E'$. A \textit{path} of length $k$ in a graph is an alternating sequence of vertices and edges, $v_{0}, e_{0}, v_{1}, e_{1}, v_{2},... , v_{k-1}, e_{k-1}, v_{k}$, where $v_{i}'$s are distinct and $e_{i}$ is the edge joining $v_{i}$ and $v_{i+1}$. This is called a path joining $v_{0}$ and $v_{k}$. A path with $v_{0} = v_{k}$ is called \textit{cycle}. A cycle of length 3 is called a \textit{triangle}. A graph is \textit{connected} if for any pair of vertices $u,v\in V$, there exists a path joining $u$ and $v$. A graph is said to be \textit{triangulated} if for any vertex $u \in V$, there exists $v,w\in V$ such that $(u,v,w)$ is a triangle. The \textit{distance} between two vertices $u, v\in V$, $d(u,v)$ is defined as the length of the shortest path joining $u$ and $v$, if it exists. Otherwise $d(u, v)$ is defined as $\infty$. The \textit{diameter} of a graph is defined as $diam(G) = max_{u,v\in V} d(u, v)$, the largest distance between pairs of vertices of the graph, if it exists. Otherwise $diam(G)$ is defined as $\infty$. The \textit{girth} of a graph is the length of its shortest cycle. The \textit{neighbourhood} of a vertex $v$ of a graph $G$ is the induced subgraph of G consisting of all vertices adjacent to $v$. We assume the neighbourhood to be closed, i.e., the vertex $v$ is included in it and denote it by $N[v]$. A vertex $v$ is called \textit{simplicial} if $N[v]$ is a clique.

Let $S$ be a non-empty set and ‘+’ and ‘.’ be two binary operations on $S$, called addition and multiplication respectively. Then $(S,+, .)$
is called a {\it semiring} if
\\(i) $(S,+)$ is a commutative semigroup,
\\(ii) $(S,.)$ is a semigroup,
\\(iii) $a.(b + c) = a.b + a.c$ and $(b +c).a = b.a + c.a$, for all $a, b, c \in S$.
\\If there exists an element $0\in S$ such that $a + 0 = a$ for all $a \in S$ then $0$ is called additive neutral element or the zero of $S$ and $S$ is called a {\it semiring with zero}.
Moreover, if $a.0 = 0.a = 0$ for all $a \in S$ then $S$ is called a {\it semiring with absorbing zero}.
Again if there exists an element $1 \in S$ such that $a.1 = 1.a = a$ for all $a \in S$ then 1 is called multiplicative identity or simply
identity element of $S$ and $S$ is called a {\it semiring with identity}.
\\Further if $a.b = b.a$ for all $a, b \in S$ then $S$ is called a {\it commutative semiring}.\\
An {\it ideal} $I$ of a semiring $S$ is a nonempty subset of $S$ such that $a + b, as, sa\in I$ for all $a,b\in I$ and $s\in S$.\\
A proper ideal $I$ of a semiring $S$ is called {\it prime} if $ab\in I$ implies $a\in I$ or $b\in I$ for any elements $a, b\in S$.\\
A proper ideal $I$ of a semiring $S$ is called {\it maximal} if it is not contained in any other proper ideal of the semiring $S$.

We refer to the book of Golan \cite{Golan} for more definitions and results of semirings and mention few results that are taken from \cite{Vechtomov} and which will be used later.

\begin{remark}\cite{Vechtomov}
Let $C_{+}(X)$ be the set of all non-negative valued continuous functions over a topological space $X$. This set with pointwise addition and multiplication forms a semiring. It is easy to see that the ring $C(X)=C_{+}(X)-C_{+}(X)$ is a ring of differences of the semiring $C_{+}(X)$ and the semiring $C_{+}(X)$ coincides with the set of all squares of elements of the ring $C(X)$.

\end{remark}

\begin{proposition}\label{maximal ideal in C+(X)}\cite{Vechtomov}
Prime (maximal) ideals of the semiring $C_{+}(X)$ are precisely the ideals $P\cap C_{+}(X)$ for the prime (maximal) ideals $P$ of the ring $C(X)$.

\end{proposition}

\begin{proposition}\cite{Vechtomov}
For any Tychonoff space $X$, the maximal ideals of the semiring $C_{+}(X)$ coincide with the ideals of the form $M^{p}=\{f\in C_{+}(X) : p\in
\overline{Z(f)}_{\beta X}\}, p\in \beta X$.
\end{proposition}

\begin{proposition}\cite{Vechtomov}
Any prime ideal of a semiring $C_{+}(X)$ is contained in a unique maximal ideal.
\end{proposition}

\begin{theorem}\cite{Ray 1995}\label{ring iso iff semiring iso}
Let the topological spaces $X$ and $Y$ be realcompact spaces. Then $X$ is homeomorphic to $Y$ if and only if $C_{+}(X)$ is isomorphic to $C_{+}(Y)$.
\end{theorem}

\section{Zero-Set Intersection Graph of $C_{+}(X)$}

In this section we initiate the study of zero-set intersection graph of the semiring $C_{+}(X)$ for a Tychonoff space $X$, by adopting the definition of the Zero-set intersection graph introduced in \cite{ref1} on $C(X)$.

First we introduce few notions on $C_{+}(X)$.

\begin{definition}\label{definition of units}
An element $f\in C_{+}(X)$ is called an \textit{unit} of $C_{+}(X)$ if there exists $g\in C_{+}(X)$ such that $fg=\mathbf{1}$, where $\mathbf{1}(x)=1$ for all $x\in X$. Equivalently, the units of $C_{+}(X)$ are all such functions that does not attain zero at any point in $X$.
\end{definition}

\begin{definition}\label{definition of graph}
Let $\mathcal{N'}(X)$ be the set of all non-units in the semiring $(C_{+}(X),+,.)$ ($\mathcal{N'}(X)=\mathcal{N}(X)\cap C_{+}(X)$, where $\mathcal{N}(X)$ is the set of all non-units in $C(X)$). By \textit{zero-set intersection graph} $\Gamma(C_{+}(X))$, we mean the graph whose set of vertices is $\mathcal{N'}(X)$ and there is an edge between distinct vertices $f$ and $g$ if $Z(f)\cap Z(g)\neq\emptyset$, where $Z(f)=\{x\in X:f(x)=0\}$.

\end{definition}

In \cite{ref1} authors observed the graph properties of the graph $\Gamma(C(X))$ of $C(X)$ viz., connectedness, diameter, girth etc. Throughout this study we have considered the graph of $C_{+}(X)$ as a subgraph of the zero-set intersection graph of $C(X)$. As a consequence, the graph $\Gamma(C_{+}(X))$ inherits some of the graph properties of $\Gamma(C(X))$ which are as follows:

\begin{itemize}
    \item[(i)]  $\Gamma(C_{+}(X))$ is {\it connected}
    \item[(ii)] $\Gamma(C_{+}(X))$ is {\it triangulated}
    \item[(iii)] $diam(\Gamma(C_{+}(X)))=2$
    \item[(iv)] {\it girth} of $\Gamma(C_{+}(X))$ is $3$
\end{itemize}

\subsection{Maximal Cliques in $\Gamma(C_{+}(X))$:}

In this subsection we study the cliques and maximal cliques in $\Gamma(C_{+}(X))$ and their relations with ideals, maximal ideals of $C_{+}(X)$. We characterize the maximal cliques in $\Gamma(C_{+}(X))$ and study some results on those.

\begin{theorem}\label{clique obs}
(1) Every ideal in the semiring $C_{+}(X)$ is a clique in $\mathcal N'(X)$.
\\(2) If $A$ is a clique in $\Gamma(C(X))$ then $A\cap C_{+}(X)$ is a clique in $\Gamma(C_{+}(X))$.
\\(3) If $A$ is a clique in $\Gamma(C_{+}(X))$ then $A$ is a clique in $\Gamma(C(X))$.

\end{theorem}

\begin{proof}
(1) If $I$ is an ideal of $C_{+}(X)$ and any two function $f,g\in I$. Then $Z(f)\cap Z(g)=Z(f+g)$. It shows that $I$ is a clique in $\mathcal N'(X)$. \\
(2) and (3) are obvious.

\end{proof}

\begin{theorem}
No maximal clique in $\Gamma(C_{+}(X))$ is a maximal clique in $\Gamma(C(X))$.    
\end{theorem}

\begin{proof}
For every continuous function $f$ belonging to a maximal clique $M$ in $\Gamma(C_{+}(X))$, the function $(-f)(\in C(X))$ does not belong to $M$, where $(-f)(x)=-f(x),x\in X$. So $M\cup \{-f\}$ is again a clique in $\Gamma(C(X))$ containing $M$. Hence the result.
\end{proof}

From Theorem \ref{clique obs} (2) we see that for any maximal clique $M$ in $\Gamma(C(X))$, $M\cap C_{+}(X)$ is a clique in $\Gamma(C_{+}(X))$. The next result additionally shows that $M\cap C_{+}(X)$ is a maximal clique in $\Gamma(C_{+}(X))$ as well.

\begin{theorem}\label{max1}
If $M$ is a maximal clique in $\Gamma(C(X))$ then $M\cap C_{+}(X)$ is a maximal clique in $\Gamma(C_{+}(X))$.
\end{theorem}

\begin{proof}
Let $M$ be a maximal clique in $\Gamma(C(X))$. Then $M\cap C_{+}(X)$ is a clique in $\Gamma(C_{+}(X))$. Let us suppose that $M\cap C_{+}(X)$ is not a maximal clique in $\Gamma(C_{+}(X))$. Then there exists $f\in C_{+}(X)\setminus M$ such that $f$ is adjacent to every element of $M\cap C_{+}(X)$. Now, for all $h\in M$, $h^{^2}\in M\cap C_{+}(X)$ and also $Z(h)=Z(h^{2})$. Therefore if for all $g\in M\cap C_{+}(X)$, $Z(f)\cap Z(g)\neq\emptyset$ then it implies that $Z(f)\cap Z(h)\neq\emptyset$ for all $h\in M$. Hence it follows that $f$ is adjacent to every element of $M$. So by maximality of $M$, $f\in M$, which is a contradiction to our assumption. Therefore $M\cap C_{+}(X)$ is a maximal clique in $\Gamma(C_{+}(X))$. 
\end{proof}

In Remark 3.4 of \cite{ref1}, the classification of the maximal cliques in $\Gamma(C(X))$ into three different categories, viz., fixed ideal, free ideal and non-ideal form, has been made. Here in the following theorem we characterize the precise form of the maximal cliques in $\Gamma(C_{+}(X))$ via the maximal cliques in $\Gamma(C(X))$, applying which it is easy to classify the maximal cliques in $\Gamma(C_{+}(X))$ as well.

\begin{theorem}\label{max2}
Any maximal clique in $\Gamma(C_{+}(X))$ is of the form $M\cap C_{+}(X)$, where $M$ is a maximal clique in $\Gamma(C(X))$.
\end{theorem}

\begin{proof}
Let $N$ be a maximal clique in $\Gamma(C_{+}(X))$. Then $N$ is also a clique in $\Gamma(C(X))$ which is contained in some maximal clique $M$ in $\Gamma(C(X))$ (say). Therefore $N$ is contained in $M\cap C_{+}(X)$, which is a maximal clique in $\Gamma(C_{+}(X))$ by Theorem \ref{max1}. Hence $N=M\cap C_{+}(X)$.
\end{proof}

We have seen that ideals of $C_{+}(X)$ are cliques in $\Gamma(C_{+}(X))$ (Theorem \ref{clique obs}). Also we obtain the maximal cliques of $\Gamma(C_{+}(X))$ from the maximal cliques of $\Gamma(C(X))$. In the next theorem we find that the maximal ideals of $C_{+}(X)$ are maximal cliques in $\Gamma(C_{+}(X))$.

\begin{theorem}\label{max3}
Maximal ideals of $C_{+}(X)$ are maximal cliques in $\Gamma(C_{+}(X))$.
\end{theorem}

\begin{proof} Clearly every maximal ideal of $C_{+}(X)$ is a clique. Let $P$ be a maximal ideal of $C_{+}(X)$. Then $P=M\cap C_{+}(X)$, where $M$ is a maximal ideal of $C(X)$ (see Proposition \ref{maximal ideal in C+(X)}). Therefore using Theorem 3.2 of \cite{ref1} we have, $M$ is a maximal clique in $\Gamma(C(X))$, whence it follows that $M\cap C_{+}(X)=P$ is a maximal clique in $\Gamma(C_{+}(X))$ by Theorem \ref{max1}. Hence the proof.
\end{proof}

The converse of the above theorem is not true in general. The following example exhibit that.

\begin{example}\label{exm1}
Let $X =\beta \mathbb{N}$ with its usual topology. Let $M_{1,2}, M_{2,3}, M_{1,3}$ be the collection of all real valued continuous functions which vanish at $\{1, 2\}, \{2, 3\},$
$\{1, 3\}$ respectively. Now if we denote $M^{+}_{1,2}=M_{1,2}\cap C_{+}(X)$, $M^{+}_{2,3}=M_{2,3}\cap C_{+}(X)$, $M^{+}_{1,3}=M_{1,3}\cap C_{+}(X)$ then $M^{+} = M^{+}_{1,2}\cup M^{+}_{2,3}\cup M^{+}_{1,3}$ is a maximal clique which is not a maximal ideal in $C_{+}(X)$.  
\end{example}

Every clique is contained in a maximal clique. But that may not be unique. The following example shows that.

\begin{example}
If we take the same $X$ as Example \ref{exm1} then $M^{+}_{1,2}$ is a clique contained in two different maximal cliques $M^{+}_{1} = M^{+}_{1,2}\cup M^{+}_{2,3}\cup M^{+}_{1,3}$ and 
\\$M^{+}_{2} = M^{+}_{1,2}\cup M^{+}_{2,4}\cup M^{+}_{1,4}$, where $M^{+}_{i,j}$ is the collection of all non-negative valued continuous function which vanishes at $\{i, j\}$, for $i,j=1,2,3,4$.
\end{example}

The following Theorems \ref{max5}, \ref{max6} are the $\Gamma(C_{+}(X))$ counterparts of Theorems 3.6 and 3.11 of \cite{ref1} and deal with the characterization of the structure of the maximal cliques in $\Gamma(C_{+}(X))$ via the maximal ideals of $C_{+}(X)$.

\begin{theorem}\label{max5}
Let $M$ be a maximal clique in $\Gamma(C_{+}(X))$. Then $M$ always contains an ideal of $C_{+}(X)$.
\end{theorem}

\begin{proof}
If $M$ is a maximal clique in $\Gamma(C_{+}(X))$ then it follows from Theorem \ref{max2} that $M=N\cap C_{+}(X)$ for some maximal clique $N$ in $\Gamma(C(X))$. So $N$ contains an ideal $I$ (say) of $C(X)$ by Theorem 3.6 of \cite{ref1}. Therefore it implies that $N\cap C_{+}(X)=M$ contains the ideal $I\cap C_{+}(X)$ of $C_{+}(X)$. Hence the theorem.
\end{proof}

\begin{theorem}\label{max6}
Every maximal clique in $\Gamma(C_{+}(X))$ can be expressed as union of intersection of some maximal ideals in $C_{+}(X)$.
\end{theorem}

\begin{proof}
Let $M$ be a maximal clique in $\Gamma(C_{+}(X))$. Then $M=N\cap C_{+}(X)$, for some maximal clique $N$ in $\Gamma(C(X))$ (see Theorem \ref{max2}). So by Theorem 3.11 of \cite{ref1}, $N$ can be expressed as union of intersection of some maximal ideals in $C(X)$, i.e., $N=\bigcup_{i}\bigcap_{\alpha_{i}} N_{\alpha_{i}}$, where each $N_{\alpha_{i}}$ is a maximal ideal in $C(X)$. Therefore $M=N\cap C_{+}(X) =\bigcup_{i}\bigcap_{\alpha_{i}} (N_{\alpha_{i}}\cap C_{+}(X))$, where each $N_{\alpha_{i}}\cap C_{+}(X)$ is a maximal ideal in $C_{+}(X)$ (see Theorem \ref{maximal ideal in C+(X)}). This completes the proof.
\end{proof}

\subsection{Prime ideals and Prime cliques:}

Every prime ideal in $C(X)$ is always contained in an unique maximal ideal which is true for $C_{+}(X)$ too (cf. Proposition 2.4 of \cite{Vechtomov}). Again prime ideals of $C_{+}(X)$ are cliques (in Theorem \ref{prime2}, we prove that the prime ideals are prime cliques (Definition \ref{def. of prime clique})) in $\Gamma(C_{+}(X))$. Also the existence of non-ideal maximal cliques in $\Gamma(C_{+}(X))$ may insist a prime ideal, as a clique, to be contained in two different maximal cliques in $\Gamma(C_{+}(X))$. Therefore, these facts intrigue us to investigate whether the prime ideals in $C_{+}(X)$, as a clique, are contained in a unique maximal clique in $\Gamma(C_{+}(X))$ or not. In the next few results we are concerned about prime cliques in $\Gamma(C_{+}(X))$ and their properties.

\begin{theorem}\label{max4}
Every prime ideal in $C_{+}(X)$ is contained in a unique maximal clique in $\Gamma(C_{+}(X))$.
\end{theorem}

\begin{proof}
Let $Q$ be a prime ideal in $C_{+}(X)$. Then $Q=P\cap C_{+}(X)$, where $P$ is a prime ideal in $C(X)$ (see Proposition \ref{maximal ideal in C+(X)}). So $P$ is contained in a unique maximal clique $M$ (say), by Theorem 3.5 of \cite{ref1}. Therefore $P\cap C_{+}(X)=Q$ is contained in $M\cap C_{+}(X)$, which is a unique such maximal clique in $\Gamma(C_{+}(X))$. This completes the proof.
\end{proof}

Now we introduce the notion of the prime cliques in $\Gamma(C_{+}(X))$, adopting the same from \cite{ref1}.

\begin{definition}\label{def. of prime clique}
A clique $I$ in a graph $G$ is defined to be a {\it prime clique} if any two vertices $f,g\in G\setminus I$, which are adjacent with all elements of $I$, are adjacent to each other.
\end{definition}

The following theorem shows a connection between the prime cliques of $\Gamma(C(X))$ and those of $\Gamma(C_{+}(X))$.

\begin{theorem}\label{prime clique in C+(X)}
If $P$ is a prime clique of $\Gamma(C(X))$ then $P\cap C_{+}(X)$ is a prime clique of $\Gamma(C_{+}(X))$.
\end{theorem}

\begin{proof}
Let $f,g\in \Gamma(C_{+}(X))\setminus (P\cap C_{+}(X))$ such that $f,g$ are adjacent with all elements of $P\cap C_{+}(X)$. Then it implies that $f,g\in \Gamma(C(X))\setminus P$ and $f,g$ are adjacent with all elements of $P$. Since $P$ is a prime clique in $\Gamma(C(X))$, $f$ and $g$ are adjacent to each other. Therefore $P\cap C_{+}(X)$ is a prime clique of $\Gamma(C_{+}(X))$.
\end{proof}

The next theorem is the semiring version of Theorem 3.13 of \cite{ref1}.

\begin{theorem}\label{prime2}
Every prime ideal in $C_{+}(X)$ is a prime clique.
\end{theorem}

\begin{proof}
Let $P$ be a prime ideal of $C_{+}(X)$. Then $P$ is of the form $K\cap C_{+}(X)$ for a prime ideal $K$ of $C(X)$ (cf. Theorem \ref{maximal ideal in C+(X)}). Now, by Theorem 3.13 of \cite{ref1}, $K$, being a prime ideal, is a prime clique in $\Gamma(C(X))$. Therefore in view of Theorem \ref{prime clique in C+(X)}, $P=K\cap C_{+}(X)$ is a prime clique in $\Gamma(C_{+}(X))$. This completes the proof.
\end{proof}

Recall that in \cite{ref1} authors have constructed prime cliques in terms of maximal ideals of $C(X)$. They considered for any maximal clique $M$ of $\mathcal{N}(X)$, \\$O_{M}=\bigcup_{N\in\mathcal{A}}(\bigcap_{M^{p}\in \mathcal{N'}} O^{p})$, where $\mathcal{A}$ is the collection of maximal ideals contained in $M$ and for each $N\in\mathcal{A}$, $\mathcal{N'}$ is the collection of all maximal ideals of $C(X)$ in which $N$ can be extended and $O^{p}$ is the ideal consisting of all $f\in C(X)$ such that $cl_{\beta X} Z(f)$ is a neighbourhood of $p\in\beta X$.

They proved that for each maximal clique $M$ of $\mathcal{N}(X)$, $O_{M}\subseteq M$ and hence $O_{M}$ is a clique (cf. Theorem 3.14 of \cite{ref1}). Not only that also they proved that for every maximal clique $M$ of $\mathcal{N}(X)$, $O_{M}$ is a prime clique (cf. Theorem 3.16 of \cite{ref1}).
\\Let us define $O^{+}_{M'}=O_{M}\cap C_{+}(X)=\bigcup_{N\in\mathcal{A}}(\bigcap_{M^{p}\in \mathcal{N'}} (O^{p}\cap C_{+}(X)))$, 
\\where $M'=M\cap C_{+}(X)$ is a maximal clique in $\mathcal{N'}(X)$, with the intention to obtain $O^{+}_{M'}$ as an example of prime clique for each maximal clique $M'$ in $\mathcal{N'}(X)$.

\begin{theorem}\label{prime clique example}
For every maximal clique $M'$ in $\mathcal{N'}(X)$, $O^{+}_{M'}$ is a prime clique in $\Gamma (C_{+}(X))$.
\end{theorem}

\begin{proof}
Let $M'$ be a maximal clique in $\mathcal{N'}(X)$. Then by Theorem \ref{max2}, $M'=M\cap C_{+}(X)$, where $M$ is a maximal clique in $\mathcal{N}(X)$. Now by Theorem 3.14 of \cite{ref1}, $O_{M}$ is a clique in $\mathcal{N}(X)$. So $O_{M}\cap C_{+}(X)$ is a clique in $\mathcal{N'}(X)$ (by Theorem \ref{clique obs} (2)). Again by Theorem 3.16 of \cite{ref1}, $O_{M}$ is a prime clique in $\mathcal{N}(X)$. Therefore using Theorem \ref{prime clique in C+(X)} we conclude that $O^{+}_{M'}$ is a prime clique, for every maximal clique $M'$ in $\mathcal{N'}(X)$.
\end{proof}

At this point one may think that only prime ideals are the examples of prime cliques. But there are plenty of prime cliques which are not ideals. Following example establishes our claim.

\begin{example}
Consider the set $O^{+}_{M'}$ in $C_{+}(X)$, for a maximal clique $M'$ in $\mathcal{N'}(X)$, defined above. $O^{+}_{M'}$ is not an ideal of $C_{+}(X)$. 

But by Theorem \ref{prime clique example}, $O^{+}_{M'}$ is a prime clique in $\mathcal{N'}(X)$. 
\end{example}

\subsection{Neighbourhood Properties of zero set intersection graph:}
In this subsection we explore some properties of the neighbourhood of a vertex of the graph $\Gamma(C_{+}(X))$ and also its connection with the semiring properties of $C_{+}(X)$ and topological properties of $X$.
\\The following result establishes a relation between graph neighbourhood of a vertex of $\Gamma(C_{+}(X))$ and zero sets of the corresponding function of $C_{+}(X)$.

\begin{lemma}\label{neighbourhood property}
For any two $f,g\in\mathcal{N'}(X)$, $N[f]\subseteq N[g]$ if and only if $Z(f)\subseteq Z(g)$, where $N[f]$ represents the closed neighbourhood of $f$ in $\Gamma(C_{+}(X))$.
\end{lemma}

\begin{proof}
Let $N[f]\subseteq N[g]$. If possible, let $Z(f)\nsubseteq Z(g)$. Then there exists $p\in Z(f)$ such that $p\notin Z(g)$. Since $X$ is completely regular, there exists $h\in C(X)$ such that $h(p) = 0$ and $h(Z(g)) = 1$, meaning that $h^{2}$ is adjacent to $f$ but not to $g$, i.e., $h^{2}\in N[f]$ but $h^{2}\notin N[g]$ which contradicts the fact that $N[f]\subseteq N[g]$. Therefore $Z(f)\subseteq Z(g)$. 

Conversely, let $Z(f)\nsubseteq Z(g)$. Let $x\in N[f]$. Then $Z(x)\cap Z(f)\neq\emptyset$ implies $Z(x)\cap Z(g)\neq\emptyset$ whence it follows that $x\in N[g]$. Hence $N[f]\subseteq N[g]$. 
\end{proof}

As a consequence of Lemma \ref{neighbourhood property}, we have the following theorem which establishes that the graph isomorphism preserves the neighbourhood properties of the vertices of the graph $\Gamma(C_{+}(X))$.

\begin{theorem}\label{4.2}
If $\phi:\Gamma(C_{+}(X))\rightarrow\Gamma(C_{+}(Y))$ is a graph isomorphism then for any $f,g\in C_{+}(X)$, $Z(f)\subseteq Z(g)$ if and only if $Z(\phi(f))\subseteq Z(\phi(g))$.
\end{theorem}

\begin{proof}
Let $Z(f)\subseteq Z(g)$. Then by Lemma \ref{neighbourhood property}, $N[f]\subseteq N[g]$. We are to show that $N[\phi(f)]\subseteq N[\phi(g)]$. For this, let us suppose that for $x\in \Gamma(C_{+}(X))$, $\phi(x)\in N[\phi(f)]$, i.e., $\phi(x)$ is adjacent to $\phi(f)$. Since $\phi$ is a graph isomorphism, $x\in N[f]\subseteq N[g]$ which again implies that $\phi(x)$ is adjacent to $\phi(g)$. Therefore $\phi(x)\in N[\phi(f)]$. Hence it follows that $N[\phi(f)]\subseteq N[\phi(g)]$. By reversing the argument we can similarly prove the converse.
\end{proof}

In the next theorem we characterize simplicial property of a vertex in $\Gamma(C_{+}(X))$ for first countable topological space $X$.

\begin{theorem}
If X is first countable then a vertex $f\in\mathcal{N'}(X)$ is simplicial if and only if $Z(f)$ is singleton.
\end{theorem}

\begin{proof}
Suppose $f\in\mathcal{N'}(X)$ is simplicial, i.e., $N[f]$ is a clique. If not and if possible suppose $Z(f)$ is not a singleton set. Then there exists at least two distinct points $p,q\in Z(f)$ and hence due to first countability and complete regularity of $X$, there exists $g,h\in C_{+}(X)$ such that $Z(g) = \{p\}$, $g(q) = 1$ and $Z(h) = \{q\}$, $h(p) = 1$. Therefore $g,h\in N[f]$ though $g$ and $h$ are not adjacent, which contradicts our initial assumption. The converse part follows trivially.
\end{proof}

In the next theorem, using the neighbourhood property, we find that for any prime clique there exists a maximal clique containing it.

\begin{theorem}
For a given prime clique $P$ of $\Gamma(C_{+}(X))$, $M =\bigcap\{N[f]:f\in P\}$ is a maximal clique containing $P$.
\end{theorem}

\begin{proof}
$P\subseteq M$ is trivial. First we prove that $M$ is a clique. Let $h,k\in M$. That means $h,k$ are adjacent with all the elements of $P$. Since $P$ is a prime clique then there is an edge between $h$ and $k$. For maximality, let $M\subset N$ where $N$ is a maximal clique containing $P$. Let $g\in N\setminus M$. Then $g\notin N[f]$ for some $f\in P$. This contradicts that $N$ is a clique containing $P$. Hence the theorem follows.
\end{proof}

\section{$\Gamma(C_{+}(X))$ and Graph Isomorphisms}
In this section for two topological spaces $X$ and $Y$, we study the inter-relationships between graph isomorphisms of $\Gamma(C_{+}(X))$ and $\Gamma(C_{+}(Y))$, semiring isomorphisms of $C_{+}(X)$ and $C_{+}(Y)$ and homeomorphisms of $X$ and $Y$. Also we show that $\Gamma(C_{+}(X))$ and $\Gamma(C_{+}(Y))$ are graph isomorphic if and only if $\Gamma(C(X))$ and $\Gamma(C(Y))$ are isomorphic as graphs (Theorem \ref{small iff big graph isomorphism}). Finally, we show that graph isomorphisms of $\Gamma(C_{+}(X))$ is
equivalent to semiring isomorphisms of $C_{+}(X)$ as well as homeomorphism of $X$.

\begin{theorem}\label{5.10}
Let $X$ and $Y$ be two topological spaces such that $C_{+}(X)$ and $C_{+}(Y)$ are isomorphic as semirings. Then $\Gamma(C_{+}(X))$ and $\Gamma(C_{+}(Y))$ are graph isomorphic.
\end{theorem}

\begin{proof}
Let $\phi:C_{+}(X)\rightarrow C_{+}(Y)$ be a semiring isomorphism. Clearly the restriction of $\phi$ on the set of non-units in $C_{+}(X)$ is also a bijection from $\Gamma(C_{+}(X))$ onto $\Gamma(C_{+}(X))$ and without loss of generality we denote it by $\phi$. The only thing left to be proved is that $\phi$ preserves adjacency. Let $f$ and $g$ be adjacent in $\Gamma(C_{+}(X))$, i.e., $Z(f)\cap Z(g)\neq\emptyset$. Since $Z(f)\cap Z(g)=Z(f+g)$ and $f+g$ is a non-unit, $\phi(f+g)$ is also a non-unit, i.e., $Z(\phi(f+g))\neq\emptyset$. Then $Z(\phi(f))\cap Z(\phi(g))= Z(\phi(f)+\phi(g))=Z(\phi(f+g))\neq\emptyset$. Similarly it can be shown that if $\phi(f)$ and $\phi(g)$ are adjacent in $\Gamma(C_{+}(Y))$ then $f$ and $g$ are adjacent in $\Gamma(C_{+}(X))$. Hence $\Gamma(C_{+}(X))$ and $\Gamma(C_{+}(Y))$ are graph isomorphic.
\end{proof}

\begin{theorem}\label{big graph to small}
Let $X,Y$ be Hewitt spaces. If $\Gamma(C(X))$ and $\Gamma(C(Y))$ are graph isomorphic then $\Gamma(C_{+}(X))$ and $\Gamma(C_{+}(Y))$ are graph isomorphic.
\end{theorem}

\begin{proof}
Let $X,Y$ be Hewitt spaces. Let $\Gamma(C(X))$ and $\Gamma(C(Y))$ be graph isomorphic. Then by Theorem 5.7 of \cite{ref1}, the ring $C(X))$ is isomorphic to the ring $C(Y)$. By Theorem \ref{ring iso iff semiring iso} and in view of the fact that $X,Y$ are homeomorphic spaces if and only if $C(X))$ is isomorphic to $C(Y)$ we deduce that the semiring $C_{+}(X)$ is isomorphic to the semiring $C_{+}(Y)$. Therefore by Theorem \ref{5.10}, $\Gamma(C_{+}(X))$ and $\Gamma(C_{+}(Y))$ are graph isomorphic.
\end{proof}

Let $f\in\Gamma(C(X))$. Then $f$ can be represented as $f=(f\vee 0)+(f\wedge 0)$, where $f\vee 0=max\{f,0\}$ and $f\wedge 0=min\{f,0\}$. We will denote $f\vee 0$ as $f^{+}$ and $f\wedge 0$ as $f^{-}$, i.e., $f=f^{+}+f^{-}$.
\\From the above definition, we observe the following facts.
\begin{lemma}\label{obs 1}
Let $f\in\Gamma(C(X))$ and $f=f^{+}+f^{-}$, where $f^{+}$, $f^{-}$ have their usual meaning as above. Then
\begin{itemize}
    \item [$(i)$] $f^{+},-f^{-}\in\Gamma(C_{+}(X))$.
     \item [$(ii)$] $N[f^{+}]\cup N[-f^{-}]=\Gamma(C_{+}(X))$.
      \item [$(iii)$] $Z(f^{+})\cup Z(-f^{-})=X$.
       \item [$(iv)$] $f^{+}.f^{-}=\mathbf{0}$.
\end{itemize}
 
\end{lemma}

\begin{lemma}\label{prop obs}
Let $f\in\Gamma(C(X))$ $(f\neq 0)$ such that $f=g+h$ satisfying (i) $g,-h\in C_{+}(X)$, (ii) $g\cdot h=\mathbf{0}$. Then $g=f^{+}$ and $h=f^{-}$.
\end{lemma}

\begin{proof}
If not and if possible let $g\neq f^{+}$ or $h\neq f^{-}$. Suppose that $g\neq f^{+}$. Then there exists $x\in X$ such that $g(x)\neq f^{+}(x)$.
So $g(x)+c=f^{+}(x)$, for some real number $c$. Then $f(x)=f^{+}(x)+f^{-}(x)=g(x)+c+f^{-}(x)$. Therefore $f(x)=g(x)+h(x)$ implies that $h(x)=f^{-}(x)+c$. Now, by our assumption that $g(x)h(x)=0$, we have, $(f^{+}(x)-c)(f^{-}(x)+c)=0$ which implies that $c(f^{+}(x)-f^{-}(x))=c^{2}$ (since $f^{+}(x)f^{-}(x)=0$). Hence it follows that either $c=0$ or $f^{+}(x)-f^{-}(x)=c$. If $c=0$ then $g(x)= f^{+}(x)$ which contradicts our assumption that $g\neq f^{+}$. Again if $f^{+}(x)-f^{-}(x)=c$ then $f^{+}(x)=g(x)+c=g(x)+f^{+}(x)-f^{-}(x)$ which implies that $g(x)=f^{-}(x)$. It implies that $g(x)\leq 0$ whence it follows that $g(x)=0$, since $g\in C_{+}(X)$. So $f(x)=h(x)\leq 0$ which implies $f^{+}(x)=0=g(x)$. This is a again a contradiction to our assumption that $g(x)+c=f^{+}(x)$. Hence $g= f^{+}$. Similarly if we assume that $h\neq f^{-}$, we will get a contradiction. Therefore $g= f^{+}$ and $h= f^{-}$.
\end{proof}

Now, we will establish the converse of Theorem \ref{big graph to small}, i.e., if $\Gamma(C_{+}(X))$ and $\Gamma(C_{+}(Y))$ are graph isomorphic then $\Gamma(C(X))$ and $\Gamma(C(Y))$ are graph isomorphic. For each graph isomorphism $\phi$ between $\Gamma(C_{+}(X))$ and $\Gamma(C_{+}(Y))$ we define $\bar{\phi}:\Gamma(C(X))\rightarrow\Gamma(C(Y))$ as follows:
$$\bar{\phi}(f):=\phi(f^{+})-\phi(-f^{-}),$$ where $f=f^{+}+f^{-}\in\Gamma(C(X))$.
Then in view of Lemma \ref{obs 1} and Lemma \ref{prop obs}, we prove the following.
\begin{lemma}\label{obs 2}
For any $f\in\Gamma(C(X))$,
\begin{itemize}
    \item [$(i)$] $N[\phi(f^{+})]\cup N[\phi(-f^{-})]=\Gamma(C_{+}(Y))$.
    \item [$(ii)$] $\phi(f^{+}).\phi(f^{-})=\mathbf{0}$.
    \item [$(iii)$] $(\bar{\phi}(f))^{+}=\phi(f^{+})$ and $(\bar{\phi}(f))^{-}=-\phi(-f^{-})$.
    \item [$(iv)$] $Z(\phi(f^{+}))\cup Z(\phi(-f^{-}))=Y$.
\end{itemize}
\end{lemma}

\begin{proof}
(i) Clearly $N[\phi(f^{+})]\cup N[\phi(-f^{-})]\subseteq\Gamma(C_{+}(Y))$. Now, let $y\in \Gamma(C_{+}(Y))$. As $\phi$ is a graph isomorphism between $\Gamma(C_{+}(X))$ and $\Gamma(C_{+}(Y))$, there exists $x\in\Gamma(C_{+}(x))$ such that $\phi(x)=y$. Now, by Lemma \ref{obs 1} (ii), $N[f^{+}]\cup N[-f^{-}]=\Gamma(C_{+}(X))$ implies $x\in N[f^{+}]\cup N[-f^{-}]$. Then $x\in N[f^{+}]$ or $x\in N[-f^{-}]$ which further implies that $\phi(x)\in N[\phi(f^{+})]$ or $\phi(x)\in N[\phi(-f^{-})]$. Hence it follows that $y\in N[\phi(f^{+})]\cup N[\phi(-f^{-})]$. Therefore $N[\phi(f^{+})]\cup N[\phi(-f^{-})]=\Gamma(C_{+}(Y))$.
\\(ii) $N[\phi(f^{+})]\cup N[\phi(-f^{-})]=\Gamma(C_{+}(Y))$ implies that $N[\phi(f^{+}).\phi(f^{-})]=\Gamma(C_{+}(Y))$. Therefore $\phi(f^{+}).\phi(f^{-})=\mathbf{0}$.
\\(iii) It follows from (ii) together with the definition of the mapping $\bar{\phi}$ and Lemma \ref{prop obs}.
\\(iv) Since $Z(\phi(f^{+}))\cup Z(\phi(f^{-}))=Z(\phi(f^{+}).\phi(f^{-}))$ then by (ii) we get $Z(\phi(f^{+}))\cup Z(\phi(-f^{-}))=Y$. 
\end{proof}

In the following results we will show that the mapping $\bar{\phi}$ defined above is a graph isomorphism between $\Gamma(C(X))$ and $\Gamma(C(Y))$.

\begin{proposition}\label{bijection}
Let $\phi:\Gamma(C_{+}(X))\rightarrow\Gamma(C_{+}(Y))$ be a graph isomorphism. Then the mapping $\bar{\phi}$ is well-defined and bijective.
\end{proposition}

\begin{proof}
Let $f,g\in\Gamma(C(X))$ such that $f=g$. Also $f=f^{+}+f^{-}$ and $g=g^{+}+g^{-}$. Then $f^{+}=g^{+}$ and $f^{-}=g^{-}$. Since $\phi$ is a well defined mapping between $\Gamma(C_{+}(X))$ and $\Gamma(C_{+}(Y))$ and $f^{+},-f^{-}\in\Gamma(C_{+}(X))$, $\phi(f^{+})=\phi(g^{+})$ and $\phi(-f^{-})=\phi(-g^{-})$. Again $(\bar{\phi}(f))^{+}=\phi(f^{+})$ and $(\bar{\phi}(f))^{-}=-\phi(-f^{-})$ for any $f\in\Gamma(C(X))$. Therefore $\bar{\phi}(f)=(\bar{\phi}(f))^{+}+(\bar{\phi}(f))^{-}=\phi(f^{+})-\phi(-f^{-})=\phi(g^{+})-\phi(-g^{-})=\bar{\phi}(g)$. Hence $\bar{\phi}$ is well defined. Reversing the above implications we can prove that $\bar{\phi}(f)=\bar{\phi}(g)$ implies $f=g$. So $\bar{\phi}$ is one-one.\\
Now, let $h\in\Gamma(C(Y))$. Then $h=h^{+}+h^{-}$, where $h^{+},-h^{-}\in\Gamma(C_{+}(Y))$. Since $\phi$ is onto, there exist $h_{1},h_{2}\in\Gamma(C_{+}(X))$ such that $\phi(h_{1})=h^{+}$ and $\phi(h_{2})=h^{-}$. Also there exists $g\in\Gamma(C(X))$ such that $g^{+}=h_{1}$. So $\phi(g^{+})=\phi(h_{1})=h^{+}$. Therefore $(\bar{\phi}(g))^{+}=\phi(g^{+})=h^{+}$. So it follows that $(\bar{\phi}(g))^{-}=-\phi(-g^{-})=h^{-}$ (i.e., $h_{2}=-g^{-}$). Hence $\bar{\phi}(g)=(\bar{\phi}(f))^{+}+(\bar{\phi}(f))^{-}=h^{+}+h^{-}=h$. Therefore $\bar{\phi}$ is onto. Hence $\bar{\phi}$ is well-defined and bijective.
\end{proof}

\begin{proposition}\label{graph homo}
Let $\phi:\Gamma(C_{+}(X))\rightarrow\Gamma(C_{+}(Y))$ be a graph isomorphism. Then the mapping $\bar{\phi}$ is a graph homomorphism.
\end{proposition}

\begin{proof}
let us assume that $f,g\in\Gamma(C(X))$ such that $f$ is adjacent to $g$. Then $Z(f)\cap Z(g)\neq\emptyset$ which implies $(Z(f^{+})\cap Z(f^{-}))\cap (Z(g^{+})\cap Z(g^{-}))\neq\emptyset$. Also $Z(
f)\cap Z(g)=Z(f^{2}+g^{2})\neq\emptyset$. So $f^{+},-f^{-},g^{+},-g^{-},f^{2}+g^{2}$ are all adjacent and $N[f^{2}+g^{2}]$ is contained in $N[f^{+}],N[-f^{-}],N[g^{+}],N[-g^{-}]$. Since $\phi$ is a graph isomorphism, by the neighbourhood property (see Theorem \ref{4.2}) it follows that $N[\phi(f^{2}+g^{2})]$ is contained in $N[\phi(f^{+})],N[\phi(-f^{-})],N[\phi(g^{+})],N[\phi(-g^{-})]$. Then $Z(\phi(f^{2}+g^{2}))\subseteq Z(\phi(f^{+}))\cap Z(-\phi(-f^{-}))\cap Z(\phi(g^{+}))\cap Z(-\phi(-g^{-}))\neq\emptyset$. It implies that $Z(\bar{\phi}(f))\cap Z(\bar{\phi}(g))\neq\emptyset$ (using Lemma \ref{obs 2} (1)). Therefore $\bar{\phi}(f)$ and $\bar{\phi}(g)$ are adjacent.\\
Now, let us assume that for $f,g\in\Gamma(C(X))$, $\bar{\phi}(f)$ and $\bar{\phi}(g)$ be adjacent, i.e., $Z(\bar{\phi}(f))\cap Z(\bar{\phi}(g))\neq\emptyset$. Then $Z((\bar{\phi}(f))^{+})\cap Z((\bar{\phi}(f))^{-})\cap Z((\bar{\phi}(g))^{+})\cap Z((\bar{\phi}(g))^{-})\neq\emptyset$ which implies $Z(\phi(f^{+}))\cap Z(-\phi(-f^{-}))\cap Z(\phi(g^{+}))\cap Z(-\phi(-g^{-}))\neq\emptyset$. Also $Z(\bar{\phi}(f))\cap Z(\bar{\phi}(g))=Z((\bar{\phi}(f))^{2}+(\bar{\phi}(g))^{2})\neq\emptyset$. So $\phi(f^{+}), \phi(-f^{-}), \phi(g^{+}), \phi(-g^{-}),$
\\$(\bar{\phi}(f))^{2}+(\bar{\phi}(g))^{2}$ are all adjacent and
$N[(\bar{\phi}(f))^{2}+(\bar{\phi}(g))^{2}]$ is contained in
\\$N[\phi(f^{+})],N[\phi(-f^{-})],N[\phi(g^{+})],N[\phi(-g^{-})]$. Since $\phi$ is a graph isomorphism, by the neighbourhood property it follows that $N[\phi^{-1}((\bar{\phi}(f))^{2}+(\bar{\phi}(g))^{2})]$ is contained in $N[f^{+}],N[-f^{-}],N[g^{+}],N[-g^{-}]$. Then $Z(f^{+})\cap Z(f^{-})\cap Z(g^{+})\cap Z(g^{-})\neq\emptyset$. This implies $Z(f)\cap Z(g)\neq\emptyset$. So $f,g$ are adjacent to each other. Hence $\bar{\phi}$ is a graph homomorphism.
\end{proof}

Combining Propositions \ref{bijection}, \ref{graph homo}, we have the following theorem.

\begin{theorem}\label{small graph to big}
If $\Gamma(C_{+}(X))$ and $\Gamma(C_{+}(Y))$ are graph isomorphic then $\Gamma(C(X))$ and $\Gamma(C(Y))$ are graph isomorphic.
\end{theorem}
Therefore in view of Theorems \ref{big graph to small} and \ref{small graph to big} we have the following.

\begin{theorem}\label{small iff big graph isomorphism}
If $\Gamma(C_{+}(X))$ and $\Gamma(C_{+}(Y))$ are graph isomorphic if and only if $\Gamma(C(X))$ and $\Gamma(C(Y))$ are graph isomorphic.
\end{theorem}

\begin{theorem}\label{homeomorphic iff graph isomorphic}
Let $X$ and $Y$ be two Hewitt spaces. $X$ is homeomorphic to $Y$ if and only if the graph $\Gamma(C_{+}(X))$ is isomorphic to the graph $\Gamma(C_{+}(Y))$.
\end{theorem}

\begin{proof}
By Theorem 5.8 and Corollary 5.11 of \cite{ref1} it follows that $X$ is homeomorphic to $Y$ if and only if $\Gamma(C(X))$ is graph isomorphic to $\Gamma(C(Y))$. Also by Theorem \ref{small iff big graph isomorphism}, $\Gamma(C(X))$ and $\Gamma(C(Y))$ are graph isomorphic if and only if $\Gamma(C_{+}(X))$ and $\Gamma(C_{+}(Y))$ are graph isomorphic. Hence the result follows.
\end{proof}

Therefore combining Theorems \ref{small iff big graph isomorphism}, \ref{homeomorphic iff graph isomorphic} and \ref{ring iso iff semiring iso}, we have the following result, which establishes the fact that the graph structure of $\Gamma(C_{+}(X))$ is sensitive enough to distinguish the spaces belonging to the class of all Hewitt spaces.
\begin{theorem}\label{combined main result}
Let $X,Y$ be Hewitt spaces. Then the following are equivalent:
\begin{itemize}
    \item [$(i)$] $X$ is homeomorphic to $Y$.
    \item [$(ii)$] The semirings $C_{+}(X)$ and $C_{+}(Y)$ are isomorphic.
    \item [$(iii)$] $\Gamma(C_{+}(X))$ and $\Gamma(C_{+}(Y))$ are graph isomorphic.
    \item [$(iv)$] $\Gamma(C(X))$ and $\Gamma(C(Y))$ are graph isomorphic.
    \item [$(v)$] The rings $C(X)$ and $C(Y)$ are isomorphic.
\end{itemize}
\end{theorem}

\begin{tikzpicture}
    \node[
      regular polygon,
      regular polygon sides=5,
      minimum width=100mm,
    ] (PG) {}
      (PG.corner 1) node (PG1) {$X\cong Y$ (as hewitt spaces)}
      (PG.corner 2) node (PG2) {(as graphs) $\Gamma(C(X))\cong \Gamma(C(Y))$}
      (PG.corner 3) node (PG3) {(as rings)$C(X)\cong C(Y)$}
      (PG.corner 4) node (PG4) {$C_{+}(X)\cong C_{+}(Y)$ (as semirings)}
      (PG.corner 5) node (PG5) {$\Gamma(C_{+}(X))\cong \Gamma(C_{+}(Y))$ (as graphs)}
    ;
    \foreach \S/\E in {
      1/3, 1/4, 1/5,
      2/1, 2/3, 2/4, 2/5,
      3/4,
      5/3, 5/4%
    } {
      \draw[thick=3,double,<->] (PG\S) -- (PG\E);
    }
  \end{tikzpicture}
\begin{center}
      The Combined Picture
\end{center}

From the following remark we will see that there exists subcollections of $C(X)$ having a more generalized algebraic structure than the semiring $C_{+}(X)$ that are graph isomorphic to the graph $\Gamma(C_{+}(X))$ of $C_{+}(X)$. Moreover, those structures also preserve the graph isomorphisms along with the homeomorphism of topological spaces.
\begin{remark}
Let us consider the set $C_{-}(X)$ of all non-positive valued continuous functions over a topological space $X$. It is no longer a semiring, rather takes home in a more generalized algebraic structure called $\Gamma$-semiring with pointwise addition and multiplication, where $\Gamma$ is the same set $C_{-}(X)$ (see section 5 of \cite{paper 3} for further details on the $\Gamma$-semiring $C_{-}(X)$). Now, if we consider the induced subgraph on $C_{-}(X)$ then we can easily observe that the two subgraphs $\Gamma(C_{+}(X))$ and $\Gamma(C_{-}(X))$ of the graph $\Gamma(C(X))$ are isomorphic. Also from the results proved in \cite{paper 3} it follows that the two semirings $C_{+}(X)$ and $C_{+}(Y)$ are isomorphic if and only if the two $\Gamma$-semirings $C_{-}(X)$ and $C_{-}(Y)$ are isomorphic. Therefore we conclude the following result.
\end{remark}

\begin{theorem}
Let $X,Y$ be Hewitt spaces. Then the following are equivalent:
\begin{itemize}
    \item [$(i)$] $X$ is homeomorphic to $Y$.
    \item [$(ii)$] The $\Gamma$-semirings $C_{-}(X)$ and $C_{-}(Y)$ are isomorphic. 
    \item [$(iii)$] $\Gamma(C_{-}(X))$ and $\Gamma(C_{-}(Y))$ are graph isomorphic.
    \item [$(iv)$] The semirings $C_{+}(X)$ and $C_{+}(Y)$ are isomorphic.
    \item [$(v)$] $\Gamma(C_{+}(X))$ and $\Gamma(C_{+}(Y))$ are graph isomorphic.
    \item [$(vi)$] $\Gamma(C(X))$ and $\Gamma(C(Y))$ are graph isomorphic.
    \item [$(vii)$] The rings $C(X)$ and $C(Y)$ are isomorphic.
\end{itemize}

\end{theorem}

\textbf{Concluding Remark.} To extend our work, it would be nice if one can try to characterize those subcollections of functions (not necessarily forming any particular algebraic structure like rings or semirings) from $\Gamma(C(X))$, which preserves the graph structure of $\Gamma(C(X))$. To be specific, let $A(X)$ and $A(Y)$ be two subcollections of functions of $\Gamma(C(X))$ and $\Gamma(C(Y))$ respectively. Then we are two find those $A(X)$ and $A(Y)$ for which $\Gamma(C(X))$ and $\Gamma(C(Y))$ are graph isomorphic if and only if $A(X)$ and $A(Y)$ are graph isomorphic.

\textbf{Acknowledgement.}
The authors are thankful to Prof. Angsuman Das, Department of Mathematics, Presidency University and Prof. Sujit Kumar Sardar, Department of Mathematics, Jadavpur University, for their encouragment and their valuable suggestions to improve the paper. The first author is grateful to Department of Science and Technology, Govt. of India, for providing research fellowship as an SRF.

\end{document}